\documentclass[11pt]{amsart}
\usepackage{amsfonts}
\usepackage{xy}
\usepackage{epic, eepic}
\usepackage{amsmath,amstext,amsbsy,amssymb}
\usepackage{color}

\newtheorem{theorem}{Theorem}[section]

\newtheorem{proposition}[theorem]{Proposition}
\newtheorem{corollary}[theorem]{Corollary}

\theoremstyle{definition}

\theoremstyle{remark}

\numberwithin{equation}{section}

\begin{document}
\title[Bosonic representations of the toroidal superalgebras of type $D$]{Bosonic vertex representations of the toroidal superalgebras in type $D(m,n)$}
\author{Naihuan Jing}
\address{Department of Mathematics, North Carolina State University, Raleigh,
NC 27695, USA} \email{jing@math.ncsu.edu}
\author{Chongbin Xu*}
\address{
School of Mathematics \& Information, Wenzhou University,
Wenzhou, Zhejiang 325035, China} \email{xuchongbin1977@126.com}
\thanks{*Corresponding author}

 \keywords{toroidal Lie
superalgebra, vertex operators, free fields}
\subjclass[2010]{Primary: 17B60, 17B67, 17B69; Secondary: 17A45,
81R10}

\begin{abstract}  In this paper, vertex representations of the 2-toroidal Lie
superalgebras of type $D(m, n)$ are constructed using both bosonic fields
and vertex operators based on their loop algebraic presentation.
\end{abstract} \maketitle

\section{Introduction}
Let $\mathfrak{g}=\mathfrak g_{\overline{0}}+\mathfrak
g_{\overline{1}}$ be a finite dimensional complex simple Lie
superalgebra under the Lie superbracket, and let $R$ be the algebra
of  Laurent polynomials in $\nu$ commuting indeterminates. By
definition, the $\nu$-toroidal Lie superalgebra associated to
$\mathfrak{g}$ is the perfect universal central extension of the
loop Lie-superalgebra $L(\mathfrak{g})=\mathfrak{g}\otimes R$,
equivalently,  one can realize it as certain homomorphic image of
the universal central extension $T(\mathfrak{g})=\mathfrak{g}\otimes
R+ \Omega_{R}/dR$, where $\Omega_{R}/dR$ is the K\"ahler
differential of $R$ modulo the exact forms.

Representations of toroidal Lie algebras have been actively
studied and a lot of known constructions for classical affine Lie algebras \cite{FF}
 have been extended to the toroidal setting.   In \cite{MRY} Moody, Rao and Yokonuma gave
the loop algebra realization of the 2-toroidal algebras and constructed vertex representation for the simply laced types.
Vertex operator representations of toroidal Lie algebras in type $B$ were
given in \cite{T}, and then generalized to multi-loop toroidal Lie algebras of the same type in \cite{JM}.
A uniformed fermionic construction of the 2-toroidal algebras of the classical types
were given by Misra and the authors \cite{JMX} and
subsequently a general bosonic construction was realized in \cite{JX1}.
Moreover, representations of the universal toroidal Lie algebras have been studied in \cite{BB}.
A Wakimoto type realization was also given for the toroidal Lie algebra in type A in \cite{BCJ}
using noncommutative differential operators (see \cite{JMT} for an earlier construction for type $A_1$) was also given.

The study of toroidal super Lie algebras is more involved and
requires new method to treat the odd subalgebra.
Based on Kac-Wakimoto's work on the affine algebras, Rao
constructed vertex representations for the toroidal general linear
superalgebra in \cite{R}. The authors have showed a MRY-type presentation for the 2-toroidal Lie
superalgebras, and constructed the unitary and orthosymplectic series by means of free
fields in \cite{JX1} based upon the well-known constructions of Lie
superalgebras \cite{FF, F} in level one. Moreover, a new vertex
representation for 2-toroidal special linear superalgebra was also given in
\cite{JX2}.
However, it is not known if other constructions of the toroidal Lie algebras such as the generalized Feingold-Frenkel construction
given in \cite{JM}
can be lifted to the super situation.

{In this paper, we use vertex
operators to realize the even part of the orthosymplectic 
toroidal Lie algebra  
and bosonic operators for the
remaining portion and then combine these two types of operators to
construct the whole algebra. In particular, we have generalized the
level $-1$ construction of the orthogonal and symplectic affine Lie
algebras \cite{JM} to the super case. Our method is a natural
generalization of \cite{JX2} given the close relationship between
orthosymplectic superalgebras and classical Lie algebras.

 The paper is organized as follows. In section 2 we
recall the notion of 2-toroidal Lie superalgebras  of type $D(m,n)$
 and the loop-algebra presentation. In section 3 we use certain vertex
operators and Weyl bosonic fields to give a level $-1$
representation of the Lie superalgebras.

\section{Toroidal Lie  superalgebras of type $D(m,n)$}
For two fixed natural numbers $m,n\in\mathbb{N}$, let $V=V_{\overline{0}}\oplus
V_{\overline{1}}$ be the super vector space with $\mbox{dim}V_{\overline{0}}=2m,\
 \mbox{dim}V_{\overline{1}}=2n$. The super-endomorphisms of $V$
form the general linear superalgebra $\mathfrak{gl}(2m|2n)$ under
the superbracket given by
$$[f, g]=fg-(-1)^{|f|\cdot |g|}gf$$
for homogeneous
linear operators $f, g$. Let $(\ \cdot\ |\ \cdot\ )$ be a non-degenerate bilinear
form on $V$ such that $(V_{\overline{0}}|V_{\overline{1}})=0$, and the
restriction of $(\cdot|\cdot)$ to $V_{\overline{1}}$ is symmetric
and the restriction to $V_{\overline 0}$ is skew-symmetric. For
$\overline{\alpha}=0,1$, let
\begin{align*}
&\mathfrak{osp}(2m|2n)_{\overline{\alpha}}\\
&=\big\{a\in\mathfrak{gl}(2m|2n)_{\overline{\alpha}}\big|(a(x)|y)
+(-1)^{\alpha p(x)}(x|a(y))=0,x,y\in V \big\}
\end{align*}
and
$\mathfrak{osp}(2m|2n)=\mathfrak{osp}(2m|2n)_{\overline{0}}\oplus
\mathfrak{osp}(2m|2n)_{\overline{1}}$. Then $\mathfrak{osp}(2m|2n)$
forms the Lie superalgebra of type $D(m,n)$ if $m>1$. Note that
\begin{eqnarray*}
   D(m|n)_{\overline{0}} \cong
\mathfrak{so}(2m)\oplus\mathfrak{sp}(2n).
\end{eqnarray*}

Let us denote the Lie  superalgebra $D(m,n)$ by  $\mathfrak{g}$, and
let $R=\mathbb C[s^{\pm1},t^{\pm1}]$ be the complex commutative ring
of Laurant polynomials in $s, t$. The loop Lie superalgebra
$L(\mathfrak{g}):=\mathfrak{g}\otimes R$ is defined under the Lie
superbracket 
$$ [x\otimes a, y\otimes b]=[x, y]\otimes ab. $$

Let $\Omega_R$ be the $R$-module of K\"ahler differentials $\{bda
|a, b\in R\}$ and $d\Omega_R$ be the space of exact forms. The
quotient space $\Omega_R/d\Omega_R$ has a basis consisting of
$\overline{s^{-1}t^{k}ds}$, $\overline{s^{l}t^{-1}dt}$,
where $k, l\in\mathbb Z$. Here $\overline{a}$
denotes the coset $a+d\Omega_R$. The toroidal superalgebr
$T(\mathfrak{g})$ is defined to be the Lie superalgebra on the
following vector space:
\begin{equation*}
T({\mathfrak{g}})=\mathfrak{g}\otimes R\oplus \Omega_R/d\Omega_R
\end{equation*}
with the Lie superbracket ($x,y\in \mathfrak{g},~a,b\in R$):
$$[x\otimes a, y\otimes b]=[x, y]\otimes ab+(x|y)\overline{(da)b},
\quad  \Omega_R/d\Omega_R ~\mbox{is central}
$$ and the parity
is specified by: $$p(x\otimes a)=p(x),\quad
p(\Omega_R/d\Omega_R)=\overline{0}.$$

Let $A=(a_{ij})$ be the extended distinguished Cartan matrix of the
affine Lie superalgebra of type $D(m,n)^{(1)}$, i.e.
$$\left(
  \begin{array}{ccccccccccc}
    2 & -1 & 0 & \cdots & 0 & 0 & 0 & \cdots & 0 & 0 & 0 \\
    -2 & 2 & -1&\cdots& 0& 0 & 0 &\cdots & 0 & 0 & 0 \\
    0 & -1& 2 &\cdots & 0 & 0 & 0 & \cdots & 0 & 0 & 0 \\
    \vdots & \ddots & \ddots & \ddots&\ddots & \ddots &\ddots & \ddots& \ddots & \ddots& \vdots \\
    0 & 0 &  \ddots & -1 & 2& -1 &\ddots & 0 & 0 & 0& 0 \\
    0 & 0 & \cdots & 0 & -1 & 0 & 1 &\ddots & 0 & 0& 0 \\
    0 & 0 & \cdots& 0 & 0 &-1& 2 & -1 & 0 & 0 & 0 \\
    \vdots & \ddots& \ddots & \ddots & \ddots & \ddots & \ddots & \ddots & \ddots & \ddots & \vdots \\
    0 & 0 & 0 & \cdots& 0 & 0 & 0 & \ddots & 2 & -1 & -1 \\
    0 & 0 & 0 &\cdots & 0 & 0 & 0 & \cdots & -1 & 2 & 0 \\
    0 & 0 & 0 &\cdots & 0 & 0 & 0 & \cdots & -1 & 0 & 2 \\
  \end{array}
\right)\leftarrow (n+1)\mbox{-th}$$ and
$Q=\mathbb{Z}\alpha_{0}\oplus\cdots\oplus \mathbb{Z}\alpha_{m+n}$ be
its root lattice. The odd simple root is $\alpha_{n}$.
The standard invariant form is then given by
$(\alpha_{i},\alpha_{j})=d_{i}a_{ij}$, where
$$(d_{0},d_{1},\cdots,d_{n},d_{n+1},\cdots,d_{n+m})
=(2,\underbrace{1,\cdots,1}_{n},\underbrace{-1,\cdots,-1}_{m}).
$$
Note that $d_i=\frac12(\alpha_i, \alpha_i)$ for non-isotropic roots.

To organize the commutation relations for toroidal Lie algebras, we
use formal series. The formal delta function is defined by
$\delta(z-w)=\sum_{n\in\mathbb{Z}}z^{-n-1}w^{n}$, which can be
formally viewed as a sum of two power series expanded at opposite
directions. For this purpose we denote that
\begin{align*}i_{z,w}\frac{1}{(z-w)}&=\sum_{n=0}^{\infty}z^{-n-1}w^{n}\\
i_{w, z}\frac{1}{(z-w)}&=-\sum_{n=0}^{\infty}w^{-n-1}z^{n},
\end{align*}
where $i_{z,w}$ means that the
power series is expanded in the domain of $|z|>|w|$. Subsequently one has that \cite{K1}:
$$\partial^{(j)}_{w}\delta(z-w)=i_{z,w}\frac{1}{(z-w)^{j+1}}-i_{w,z}\frac{1}{(-w+z)^{j+1}}.$$
where $\partial^{(j)}_{w}=\partial^{j}_{w}/j!$. By convention if we write
 a rational function in the variable $z-w$
 it is usually assumed that
 the power series is expanded in the region
 $|z|>|w|$.

The following loop algebra presentation of the 2-toroidal Lie
superalgebras was proved in \cite{JX1}.
\begin{theorem} The  toroidal Lie superalgebra $T(D(m,n))$
is isomorphic to the Lie superalgebra
 $\mathfrak{T}(A)$  generated by
$$\{\mathcal{K},\alpha_{i}(k),x^{\pm}_{i}(k)|\, 0\leqslant i\leqslant
m+n,k\in\mathbb{Z}\}$$
 with parities given as : \emph{(}$0\leqslant i\leqslant m+n,k\in\mathbb{Z}$\emph{)}
$$p(\mathcal{K})=p(\alpha_{i}(k))=\overline{0},\quad p(x^{\pm}_{i}(k))=p(\alpha_{i}).$$
subject to the following relations
 \begin{eqnarray*}
    &1)& [\mathcal{K},\alpha_{i}(z)]=[\mathcal{K},x^{\pm}_{i}(z)]=0; \\
    &2)& [\alpha_{i}(z),\alpha_{j}(w)]=(\alpha_{i}|\alpha_{j})\partial_{w}\delta(z-w)\mathcal{K};\\
    &3)& [\alpha_{i}(z),x^{\pm}_{j}(w)]=\pm(\alpha_{i}|\alpha_{j})x^{\pm}_{j}(w)\delta(z-w); \\
    &4)& [x^{+}_{i}(z),x^{-}_{j}(w)]=0,\mbox{\emph{if}}~i\neq j   ;\\
    &&[x^{+}_{i}(z),x^{-}_{i}(w)]=-\{(\alpha_{i}(w)
    \delta(z-w)+\partial_{w}\delta(z-w)\mathcal{K}\},\mbox{\emph{if}}~(\alpha_{i}|\alpha_{i})=0\\
    &&[x^{+}_{i}(z),x^{-}_{i}(w)]=-\frac{2}{(\alpha_{i}|\alpha_{i})}\{(\alpha_{i}(w)
    \delta(z-w)+\partial_{w}\delta(z-w)\mathcal{K}\},\mbox{\emph{if}}~(\alpha_{i}|\alpha_{i})\neq0\\
    &5)&[x^{\pm}_{i}(z),x^{\pm}_{i}(w)]=0;\\
    &&[x^{\pm}_{i}(z),x^{\pm}_{j}(w)]=0,\mbox{\emph{if}}~a_{ii}=a_{ij}=0,i\neq j;\\
    &&[x^{\pm}_{i}(z_{1}),[x^{\pm}_{i}(z_{2}),x^{\pm}_{j}(w)]]=0,\mbox{\emph{if}}~a_{ii}=0,a_{ij}\neq0,i\neq
    j;\\
    &&[x^{\pm}_{i}(z_{1}),\cdots,[x^{\pm}_{i}(z_{1-a_{ij}}),x^{\pm}_{j}(w)]\cdots]=0,
    \mbox{\emph{if}}~a_{ii}\neq0,i\neq j,
    \end{eqnarray*}
where we have used the generating series $\alpha_{i}(z)=\sum_{k\in\mathbb{Z}}\alpha_{i}(k)z^{-k-1},
x^{\pm}_{i}(z) =\sum_{k\in\mathbb{Z}}x^{\pm}_{i}(k)z^{-k-1}$.
 \end{theorem}
Note that all brackets in the relations are understood as super-brackets.

\section{Representations of toroidal Lie superalgebras}
This section is devoted to realization of the
toroidal Lie superalgebra of $D(m,n)$ using both bosonic fields and
vertex operators.

 Let $\varepsilon_{i}\ (0\leqslant
i\leqslant n+m+1)$ be an orthomormal basis of the vector space
$\mathbb{C}^{n+m+2}$ and
$\delta_{i}=\sqrt{-1}\varepsilon_{n+i}~(1\leqslant i \leqslant
m+1)$, then the distinguished simple root systems, positive root
systems and the longest distinguished root of the Lie superalgebra of
type $D(m,n)$ can be represented in terms of vectors
$\varepsilon_{i}$'s and $\delta_{i}$'s as follows.
\begin{align*}
\triangle_{+}&=\big\{\varepsilon_{i}\pm\varepsilon_{j},
\delta_{k}\pm\delta_{l}|1\leqslant i<j\leqslant n,1\leqslant k< l\leqslant m\big\},\\
&\cup\big\{2\delta_{k},\varepsilon_{i}\pm\delta_{i}|1\leqslant
i\leqslant n,1\leqslant k\leqslant n\big\},\\
\Pi&=\big\{\alpha_{1}=\varepsilon_{1}-\varepsilon_{2},\cdots,
\alpha_{n-1}=\varepsilon_{n-1}-\varepsilon_{n},\alpha_{n}=\varepsilon_{n}-\delta_{1}, \\
&\alpha_{n+1}=\delta_{1}-\delta_{2}, \cdots,
 \alpha_{n+m-1}=\delta_{m-1}-\delta_{m}, \alpha_{n+m}=\delta_{m-1}+\delta_{m}, \\ 
\theta &=2\alpha_{1}+\cdots+2\alpha_{n+m-2}+\alpha_{n+m-1}+\alpha_{n+m}=2\varepsilon_{1}.
\end{align*}


 Let $\overline{c}=\varepsilon_{0}+\delta_{m+1}$ and define
 $\alpha_{0}=\overline{c}-\theta, \beta=-\overline{c}+\varepsilon_{1}$, then
$\alpha_{0}=-\beta-\varepsilon_{1}$. Note that
$(\beta|\beta)=1,(\beta|\varepsilon_{i})=\delta_{1,i}$. Let
$\mathcal {P}$ be the vector spaces spanned by the set
$\{\overline{c},\varepsilon_{i}|1\leqslant i\leqslant n+m\}$ and
$\mathcal {P}^{*}$ be its dual space. Let $\mathcal{C}=\mathcal
{P}\oplus \mathcal {P}^{*} $ and define the bilinear form $\langle \ ,  \ \rangle$ on $\mathcal{C}$ as
follows: $\mbox{for}~ a,b\in\mathcal {P}$
\begin{eqnarray*}
   && \langle b^{*},a\rangle=-\langle
a,b^{*}\rangle=(a,b), \quad
 \langle b,a\rangle=\langle
a^{*},b^{*}\rangle=0,
\end{eqnarray*}

Let $\mathcal {A}(\mathbb{Z}^{2n+2m+2})$ be the Weyl algebra generated
by $\{u(k)|u\in \mathcal {C},k\in\mathbb{Z}\}$ with the defining
relations:
$$[u(k),v(l)]=\langle u,v\rangle\delta_{k,-l}$$
for $u,v\in \mathcal {C}$ and $ k,l\in\mathbb{Z}$.

We define the representation space of $\mathcal
{A}(\mathbb{Z}^{2(n+m+1)})$ by
$$\mathfrak{F}=\bigotimes_{a_{i}}\Big(\bigotimes_{k\in\mathbb{Z}_{+}}\mathbb{C}[a_{i}(-k)]
\bigotimes_{k\in\mathbb{Z}_{+}}\mathbb{C}[a^{*}_{i}(-k)]\Big),
$$
where $a_{i}$ runs though a fixed basis in $\mathcal{P}$, consisting
of, say $\overline{c}$ and $\varepsilon_{i}$'s. The algebra
$\mathcal {A}(\mathbb{Z}^{2(n+m+1)})$ acts on the space by the usual
action: $a(-k)$ acts as a creation operator and $a(k)$ an
annihilation operator.

For $u\in \mathcal{C}$, we define the formal power series with
coefficients from the associative algebra $\mathcal
{A}(\mathbb{Z}^{2(n+m+1)})$:
$$u(z)=\sum_{k\in\mathbb{Z}}u(k)z^{-k-1},$$
which is a bosonic field acting on the Fock space $\mathfrak{F}$.

\begin{proposition}\cite{JMX} The bosonic fields satisfy the following
communication relation:
$$[u(z),v(w)]=\langle u,v\rangle\delta(z-w).$$
\end{proposition}
Let
$L=\mathbb{Z}\varepsilon_{1}\oplus\cdots\oplus\mathbb{Z}\varepsilon_{m+n}$
and $\mathfrak{h}=L\otimes_{Z}\mathbb{C}$ be its complex hull. We
view $\mathfrak{h}$ as an abelian Lie algebra and define its central
extension
$$\widehat{\mathfrak{h}}=\bigoplus_{k\in \mathbb{Z}}\mathfrak{h}\otimes t^{k}\oplus
\mathbb{C}\mathfrak{c}$$ with the following Lie multiplication:
$$[\alpha(k),\beta(l)]=k(\alpha,\beta)\delta_{k,-l}\mathfrak{c},\quad [\widehat{\mathfrak{h}},\mathfrak{c}]=0$$
where $\alpha(k)=\alpha\otimes t^{k}$ and $\alpha,\beta\in L; k,l\in
\mathbb{Z}$.

Let $\widehat{\mathfrak{h}}_{\pm}=\bigoplus_{k\in
\mathbb{Z}^{+}}\mathfrak{h}\otimes t^{\pm k}$ and
$S(\widehat{\mathfrak{h}}_{-})$ the symmetric algebra of
$\widehat{\mathfrak{h}}_{-}$. We give
$S(\widehat{\mathfrak{h}}_{-})$  an
$\widehat{\mathfrak{h}}_{+}\oplus\widehat{\mathfrak{h}}_{-}\oplus
\mathbb{C}\mathfrak{c}$-module structure by letting $\alpha(-k)$ act
as the multiplication by $\alpha(-k)$ for $k>0$,
$\alpha(k)$ the derivation determined by
$\alpha(k)\cdot\beta(-l)=\delta_{k,l}k(\alpha,\beta)$ for $k,l>0$
and $\mathfrak{c}$ the identity operator.

 For $i=0,1$, let
$L_{\overline{i}}=\{\alpha\in L|(\alpha,\alpha)\equiv i~
(\mbox{mod}~ 2)\}$, then $L=L_{\overline{0}}\oplus
L_{\overline{1}}$. Let $F:L\times L\rightarrow \{\pm1\}$ be the
cocycle satisfying $F(0,\alpha)=F(\alpha,0)$ and
$$F(\varepsilon_{i},\varepsilon_{j})
=\left\{
\begin{array}{ll}
1,&\mbox{if}~~i\leqslant j; \\
 -1,&\mbox{if}~~i> j.
\end{array}
\right.$$

Let $\mathbb{C}[L]$ be the complex vector space spanned by the basis
$\{e^{\gamma}|\gamma\in L\}$. We define a twisted group algebra
structure on $\mathbb{C}[L]$ by
$$e^{\alpha}e^{\beta}=F(\alpha,\beta)e^{\alpha+\beta}.$$
We define the tensor space $$V[L]=
S\big(\widehat{\mathfrak{h}}_{-}\big)\bigotimes_{\mathbb{C}}\mathbb{C}[L],$$
and define the action of $\widehat{\mathfrak{h}}$ on $V[\Gamma]$ as
follows
$$\alpha(k)\cdot(v\otimes e^{\beta})\\  \nonumber
=\delta_{k,0}(\alpha,\beta) (v\otimes
e^{\beta})+(1-\delta_{k,0})\alpha(k)\cdot v\otimes e^{\beta}
$$

Then the space $V[L]$ has a natural $\mathbb{Z}_{2}$-gradation:
$$V[L]=V[L]_{\overline{0}}\oplus V[L]_{\overline{1}}$$ where
$V[L]_{\overline{0}}$ (resp.$V[L]_{\overline{1}}$) is the vector
space spanned by $e^{\alpha}\otimes \beta\otimes t^{-j}$ with
$\alpha,\beta\in L; j\in \mathbb{Z}_{+}$ such that
$(\alpha,\alpha)\in 2\mathbb{Z}$ (resp.$(\alpha,\alpha)\in
2\mathbb{Z}+1$).

For $\alpha\in L$, we define
$$E^{\pm}(\alpha,z)=\mbox{exp}\bigg(\sum_{k\in\pm \mathbb{Z}_{+}}\frac{\alpha(k)}{k}z^{-k}\bigg)
\in \mbox{End} S(\widehat{\mathfrak{h}_{-}})[[z^{\mp 1}]]
$$
and introduce the operator $z^{\alpha(0)}$ as follows:
\begin{eqnarray*}
 && z^{\alpha(0)}(e^{\beta}\otimes\gamma\otimes t^{-j})
 =z^{(\alpha,\beta)}(e^{\beta}\otimes\gamma\otimes
 t^{-j})
 \end{eqnarray*}where for $\beta,\gamma\in L$ and $j\in
 \mathbb{Z}_{+}$.  Define the vertex operator $Y(\alpha,z)$:
 $$Y(\alpha,z)=e^{\alpha}z^{\alpha(0)}E^{-}(-\alpha,z)E^{+}(-\alpha,z)$$
and denote by
$$
X(\alpha,z)=\left\{
\begin{array}{ll}
z^{\frac{(\alpha,\alpha)}{2}}Y(\alpha,z),& \mbox{if}~\alpha\in
L_{\overline{0}};\\
Y(\alpha,z),& \mbox{if}~\alpha\in L_{\overline{1}}.
\end{array}
\right.
$$
Expand $X(\alpha,z)$ in powers of $z$:
$$X(\alpha,z)=\sum_{j\in\mathbb{Z}}X(\alpha,j)z^{-j-1}.$$
Note the components $X(\alpha, j)$ are well-defined operators.

In addition, for $\alpha\in L$, we define $$\alpha(z)=\sum_{k\in
\mathbb{Z}}\alpha(k)z^{-k-1}$$ then we have
$$[\alpha(z),\beta(w)]=(\alpha,\beta)\partial_{w}\delta(z-w), \quad\qquad
\alpha,\beta\in L.$$

\begin{proposition} On the space $V[L]$ one has that
\begin{eqnarray*}
1) &&
[X(\varepsilon_{i},z),X(\varepsilon_{j}-\varepsilon_{k},w)]=\delta_{ik}F(\varepsilon_{i},
  \varepsilon_{j}-\varepsilon_{k})X(\varepsilon_{j},w)\delta(z-w),\quad  j\neq k,\\
2) &&
[X(\varepsilon_{i},z),X(-\varepsilon_{j},w)]=\delta_{ij}F(\varepsilon_{i},
  -\varepsilon_{j})\partial_{w}\delta(z-w),\\
3)&&[X(\varepsilon_{i}-\varepsilon_{j},z),X(\varepsilon_{j}-\varepsilon_{i},w)]\\
&&\quad=F(\varepsilon_{i}-\varepsilon_{j},\varepsilon_{j}-\varepsilon_{i})
\big((\varepsilon_{i}-\varepsilon_{j})(z)\delta(z-w)+\partial_{w}\delta(z-w)\big),\\
4)&&[\alpha(z),X(\beta,w)]=(\alpha,\beta)X(\beta,z)\delta(z-w),\quad
\alpha,\beta\in L.
\end{eqnarray*}
\end{proposition}
\begin{proof}1), 2) and 4) are direct consequences of Lemma 1.8 in
\cite{R}. For 3), we refer the reader to \cite{FK}.
\end{proof}

In the following, we will  give a representation of
$\mathfrak{T}(A))$ on  the tensor space $V[L]\otimes\mathfrak{F}$.
It is easy to see that there is a $\mathbb{Z}_{2}-$gradation on this
space with the parity given by $p(e^{\alpha}\otimes x\otimes
y)=p(\alpha)$ for $\alpha\in L, x\in
S(\bigoplus_{j<0}(\mathfrak{h}\otimes t^{j}),y\in \mathfrak{F}$. The
vertex operators $X(\alpha,z),\alpha(z)$ act on the first component
and the bosonic fields $u(z)$ act on the second component. It
follows that
$$p(X(\alpha,z))=p(\alpha),\quad p(\alpha(z))=p(u(z))=\overline{0}.$$

For any two fields $a(z),b(w)$ with fixed parity, we define the
normal ordered product by:
\begin{eqnarray*}
:a(z)b(w):&=&a(z)_{+}b(w)-(-1)^{p(a)p(b)}b(w)a(z)_{-}\\
  &=&(-1)^{p(a)p(b)}:b(w)a(z):
\end{eqnarray*} where $a_{\pm}(z)$ is defined as
usual.

Furthermore, we define the contraction
 of two fields $a(z),b(w)$  by
$$\underbrace{a(z)b(w)}=a(z)b(w)-:a(z)b(w):.$$

We recall the general operator product expansion \cite{K1}. Suppose $a(z),b(w)$ are two
fields such that
$$[a(z),b(w)]=\sum_{j=0}^{N-1}c^{j}(w)\partial_{w}^{(j)}\delta(z-w),$$
where $N$ is a positive integer and  $c^{j}(w)$ are  formal
distributions in the indeterminate $z$, then we have that
$$\underbrace{a(z)b(w)}=\sum_{j=0}^{N-1}c^{j}(w)\frac{1}{(z-w)^{j+1}}.$$

\begin{corollary} For $u,v\in \mathcal{C};\alpha,\beta \in L$, one has
\begin{eqnarray*}
  1)&& \underbrace{u(z)v(w)}=<u,v>\frac{1}{z-w},\quad
\underbrace{u(z)X(\alpha ,w)}=0;\\
  2) &&\underbrace{X(\varepsilon_{i},z)X(-\varepsilon_{i} ,w)}=\frac{1}{z-w}
\end{eqnarray*}
\begin{proof} These are direct results of Proposition 3.1 and the OPE.
\end{proof}
\end{corollary}

The following well-known Wick's theorem is useful for calculating
the operator product expansions of normally ordered products of free
fields.

\begin{theorem} (\cite{K1}) Let $A^{1},A^{2},\cdots,A^{M}$ and
$B^{1},B^{2},\cdots,B^{N}$ be two collections of fields with
definite parity. Suppose these fields satisfy the following
properties:
\begin{eqnarray*}
  1) && [\underbrace{A^{i}B^{j}},Z^{k}]=0,\mbox{for all}~ i,j,k ~\mbox{and}~Z=A~or~B;\\
  2) &&[A^{i}_{\pm},B^{j}_{\pm}]=0,\mbox{for all}~ i,j.
\end{eqnarray*}
then we have that
\begin{align*}
& :A^{1}\cdots A^{M}::B^{1}\cdots B^{N}: \\=&
\sum_{s=0}^{\min\{M,N\}}\sum_{i_{1}<\dots<i_{s}\atop
j_{1}\neq\dots\neq
j_{s}}\pm\Big(\underbrace{A^{i_{1}}B^{j_{1}}}\cdots
\underbrace{A^{i_{s}}B^{j_{s}}}:A^{1}\cdots A^{M}B^{1}\cdots
B^{N}:_{(i_{1},\dots,i_{s},j_{1},\dots,j_{s})}\Big)
\end{align*}
where the subscript $(i_{1},\cdots,i_{s},j_{1},\cdots,j_{s})$ means
the fields $A^{i_{1}},\dots,A^{i_{s}}$, $B^{j_{1}}$, $\dots$,
$B^{j_{s}}$ are removed and the sign $\pm$ is obtained by the rule:
each permutation of the adjacent odd fields changes the sign.
\end{theorem}

Now we state the main result in this paper.

\begin{theorem} The  map defined below
$$x_{i}^{+}(z)\mapsto
\left\{\begin{array}{lll}
\frac{1}{2}:\beta^{*}(z)\varepsilon_{1}^{*}(z):,&i=0;\\
\sqrt{-1}:\varepsilon_{i}(z)\varepsilon^{*}_{i+1}(z):,&1\leqslant
i\leqslant
n-1;\\
:X(\varepsilon_{n+1},z)\varepsilon^{*}_{n}(z):,&i=n;\\
X(\varepsilon_{i}-\varepsilon_{i+1},z),&n+1\leqslant i\leqslant
n+m-1\\
X(\varepsilon_{n+m-1}+\varepsilon_{n+m},z),&i=m+n.
\end{array}
\right.$$
$$x_{i}^{-}(z)\mapsto
\left\{\begin{array}{lll}
\frac{1}{2}:\beta(z)\varepsilon_{1}(z):,& i=0;\\
\sqrt{-1}:\varepsilon^{*}_{i}(z)\varepsilon_{i+1}(z):,&1\leqslant
i\leqslant
n-1;\\
:X(-\varepsilon_{n+1},z)\varepsilon_{n}(z):,&i=n;\\
X(\varepsilon_{i+1}-\varepsilon_{i},z),&n+1\leqslant i\leqslant
n+m-1\\
X(-\varepsilon_{n+m-1}-\varepsilon_{n+m},z),&i=m+n.
\end{array}
\right.$$

$$\alpha_{i}(z)\mapsto
\left\{\begin{array}{lll}
-\frac{1}{2}\big(:\varepsilon_{1}(z)\varepsilon_{1}^{*}(z):+:\varepsilon_{1}(z)\beta^{*}(z):+\\\qquad\quad
:\beta(z)\varepsilon_{1}^{*}(z):+:\beta(z)\beta^{*}(z):\big),& i=0;\\
:\varepsilon_{i}(z)\varepsilon_{i}(z):-:\varepsilon_{i+1}(z)\varepsilon_{i+1}(z):,&1\leqslant i\leqslant m;\\
-\varepsilon_{n+1}(z)-:\varepsilon_{n}(z)\varepsilon_{n}^{*}(z):,&i=n;\\
(\varepsilon_{i}-\varepsilon_{i+1})(z),&n+1\leqslant i\leqslant
n+m-1\\
(\varepsilon_{n}+\varepsilon_{n+m})(z),&i=m+n.
\end{array}
\right.$$ gives rise to a level -1 representation on the space
$V[L]\otimes \mathfrak{F}$ for the  2-toroidal Lie superalgebra of
type $D(m,n)$.
\end{theorem}
\begin{proof} We prove the theorem by checking  the field
operators defined above satisfying relations $1)$
--- $5)$ listed in Proposition 2.1.

First of all, we check $4)$ and $3)$ with the help of
 Wick's theorem.
\begin{align*}
&[x_{0}^{+}(z),x_{0}^{-}(w)]\\
&=\frac{1}{4}\big((:\varepsilon_{1}\varepsilon_{1}^{*}(z):+:\varepsilon_{1}(z)\beta^{*}(z):+
:\beta(z)\varepsilon_{1}^{*}(z):+:\beta(z)\beta^{*}(z):)
\delta(z-w)\\
&\hskip 2in +2\partial_{w}\delta(z-w)\big)\\
&=-\frac{2}{(\alpha_{0},\alpha_{0})}\big(\alpha_{0}(z)\delta(z-w)+\partial_{w}\delta(z-w)\cdot
(-1)\big),
\end{align*}
\begin{eqnarray*}
[\alpha_{0}(z),x_{0}^{\pm}(w)]=0=\pm(\alpha_{0},\alpha_{0})x_{0}^{\pm}(w)\delta(z-w).
\end{eqnarray*}
For $1\leqslant i\leqslant n-1$, we have that
\begin{eqnarray*}
[x_{i}^{+}(z),x_{i}^{-}(w)]
&=&-\big((:\varepsilon_{i}(z)\varepsilon_{i}(z):-:\varepsilon_{i+1}(z)\varepsilon_{i+1}(z):)\delta(z-w)+\partial_{w}\delta(z-w)\big)\\
&=&-\frac{2}{(\alpha_{i},\alpha_{i})}\big(\alpha_{i}(z)\delta(z-w)+\partial_{w}\delta(z-w)\cdot
(-1)\big).
\end{eqnarray*}
and
$[\alpha_{i}(z),x_{i}^{\pm}(w)]=\pm(\alpha_{i},\alpha_{i})x_{i}^{\pm}(w)\delta(z-w)$.

Next, one check that
\begin{align*}
&[x_{n}^{+}(z),x_{n}^{-}(w)]_{+}\\
&=\big(:\varepsilon^{*}_{n}(z)\varepsilon_{n}(z):+:X(\varepsilon_{n+1},z)X(-\varepsilon_{n+1},z):\big)
\delta(z-w)+\partial_{w}\delta(z-w)\\
&=-\big(\alpha_{n}(z)\delta(z-w)+\partial_{w}\delta(z-w)\cdot
(-1)\big),
\end{align*}
and
\begin{eqnarray*}
[\alpha_{n}(z),x_{n}^{\pm}(w)]=0=\pm(\alpha_{n},\alpha_{n})x_{n}^{\pm}(w)\delta(z-w).
\end{eqnarray*}
For $n+1\leqslant i\leqslant m+n-1$, we have that
\begin{eqnarray*}
[x_{i}^{+}(z),x_{i}^{-}(w)]&=&-\big((\varepsilon_{i}-\varepsilon_{i+1})(z)\big)\delta(z-w)+\partial_{w}\delta(z-w)\\
&=&-\frac{2}{(\alpha_{i},\alpha_{i})}\big(\alpha_{i}(z)\delta(z-w)+\partial_{w}\delta(z-w)\cdot
(-1)\big)
\end{eqnarray*}
and
$[\alpha_{i}(z),x_{i}^{\pm}(w)]=\pm(\alpha_{i},\alpha_{i})x_{i}^{\pm}(w)\delta(z-w).
$

 For the $n+m$-th vertex, one has
\begin{eqnarray*}
[x_{n+m}^{+}(z),x_{n+m}^{-}(w)]&=&-\big((\varepsilon_{n+m-1}+\varepsilon_{n+m})(z)\big)\delta(z-w)+\partial_{w}\delta(z-w)\\
&=&-\frac{2}{(\alpha_{n+m},\alpha_{n+m})}\big(\alpha_{n+m}(z)\delta(z-w)+\partial_{w}\delta(z-w)\cdot
(-1)\big)
\end{eqnarray*}and
$[\alpha_{n+m}(z),x_{n+m}^{\pm}(w)]=\pm(\alpha_{n+m},\alpha_{n+m})x_{n+m}^{\pm}(w)\delta(z-w).
$

 For all $i\neq j$, we have $ [x_{i}^{+}(z),x_{j}^{-}(w)]=0$ and
for any unconnected vertices
\begin{eqnarray*}
[\alpha_{i}(z),x_{j}^{\pm}(w)]=0=\pm(\alpha_{i},\alpha_{j})x_{j}^{\pm}(w)\delta(z-w)
\end{eqnarray*}
All the rest can be checked by straightforward calculation, for
examples
\begin{align*}
[\alpha_{0}(z),x_{1}^{+}(w)]
&=-2\sqrt{-1}:\varepsilon_{1}(z)\varepsilon^{*}_{2}(z):\delta(z-w)\\
&=(\alpha_{0},\alpha_{1})x_{1}^{+}(w)\delta(z-w),
\end{align*} where have use the property
$\beta(z)=-c(z)+\varepsilon_{1}(z)$ and
$:c(z)\varepsilon^{*}_{2}(z):=0$
\begin{align*}
[\alpha_{n-1}(z),x_{n}^{+}(w)]
&=:X(\varepsilon_{n-1},z)\varepsilon_{n}^{*}(z):\delta(z-w)\\
&=(\alpha_{n-1},\alpha_{n})x_{n}^{+}(w)\delta(z-w)\end{align*}
 By
proposition 3.2, we have
\begin{align*}[\alpha_{n+m-2}(z),x_{m+n-1}^{+}(w)]
&=(\alpha_{m+n-2},\alpha_{m+n-1})x_{m+n-1}^{+}(w)\delta(z-w)\\
[\alpha_{n+m-2}(z),x_{m+n}^{+}(w)]
&=(\alpha_{m+n-2},\alpha_{m+n})x_{m+n}^{+}(w)\delta(z-w).
\end{align*}
and others can be proved similarly.

Secondly, we can check $2)$ case by case by using Proposition 3.2
and we include the following examples
\begin{eqnarray*}
&&[\alpha_{0}(z),\alpha_{0}(w)]=-4\partial_{w}\delta(z-w)=(\alpha_{0},\alpha_{0})\partial_{w}\delta(z-w)\cdot(-1)\\
&&[\alpha_{0}(z),\alpha_{1}(w)]=2\partial_{w}\delta(z-w)=(\alpha_{0},\alpha_{1})\partial_{w}\delta(z-w)\cdot(-1)
\end{eqnarray*}

Finally, we proceed to check the Serre relations. It is easy to
verify that $[x_{i}^{\pm}(z),x_{i}^{\pm}(w)]=0$ for $0\leqslant
i\leqslant m+n$ and $[x_{i}^{\pm}(z),x_{j}^{\pm}(w)]=0$ for $ i\neq
j,a_{ij}=0$. The rest can be checked directly:
\begin{eqnarray*}
  &&[x_{0}^{+}(z_{1}),[x_{0}^{+}(z_{2}),x_{1}^{+}(w)]]\\
  &=&\frac{\sqrt{-1}}{4}[:\beta^{*}(z_{1})\varepsilon_{1}^{*}(z_{1}):,[:\beta^{*}(z_{2})\varepsilon_{1}^{*}(z_{2}):,
  :\varepsilon_{1}(w)\varepsilon^{*}_{2}(w):]]\\
   &=&\frac{\sqrt{-1}}{4}[:\beta^{*}(z_{1})\varepsilon_{1}^{*}(z_{1}):,:\varepsilon^{*}_{1}(w)\varepsilon^{*}_{2}(w):]\delta(z_{2}-w)\\
   &=&0
\end{eqnarray*}
\begin{eqnarray*}
  &&[x_{1}^{+}(z_{1}),[x_{1}^{+}(z_{2}),[[x_{1}^{+}(z_{3}),x_{0}^{+}(w)]]]\\
  &=&\frac{\sqrt{-1}}{2}[:\varepsilon_{1}(z_{1})\varepsilon_{2}^{*}(z_{1}):,[:\varepsilon_{1}(z_{2})\varepsilon_{2}^{*}(z_{2}):,
  [:\varepsilon_{1}(z_{3})\varepsilon_{2}^{*}(z_{3}):,  :\beta^{*}(w)\varepsilon^{*}_{1}(w):]]]\\
  &=&-\frac{\sqrt{-1}}{2}[:\varepsilon_{1}(z_{1})\varepsilon_{2}^{*}(z_{1}):,[:\varepsilon_{1}(z_{2})\varepsilon_{2}^{*}(z_{2}):,
  ,:\varepsilon_{2}^{*}(z_{3})\varepsilon^{*}_{1}(z_{3}):]]\delta(z_{3}-w)\\
   &=&\frac{\sqrt{-1}}{2}[:\varepsilon_{1}(z_{1})\varepsilon_{2}^{*}(z_{1}):,:\varepsilon_{2}^{*}(z_{2})\varepsilon_{2}^{*}(z_{2})]\delta(z_{2}-z_{3})\delta(z_{3}-w)\\
   &=&0
\end{eqnarray*}
\begin{eqnarray*}
  &&[x_{m+1}^{+}(z_{1}),[x_{m+1}^{+}(z_{2}),x_{m}^{+}(w)]]\\
  &=&[:X(\varepsilon_{m+1},z_{1})\delta^{*}_{1}(z_{1}):,[:X(\varepsilon_{m+1},z_{2})\delta^{*}_{1}(z_{2}):,X(\varepsilon_{m}-\varepsilon_{m+1},w)]]\\
   &=&[:X(\varepsilon_{m+1},z_{1})\delta^{*}_{1}(z_{1}):,;X(\varepsilon_{m},w)\delta_{1}^{*}(w):]\delta(z_{2}-w)\\
   &=&0
\end{eqnarray*}
\begin{eqnarray*}
  &&[x_{m+1}^{+}(z_{1}),[x_{m+1}^{+}(z_{2}),x_{m+2}^{+}(w)]]\\
  &=&\sqrt{-1}[:X(\varepsilon_{m+1},z_{1})\delta^{*}_{1}(z_{1}):,[:X(\varepsilon_{m+1},z_{2})\delta^{*}_{1}(z_{2}):
  ,:\delta_{1}(w)\delta^{*}_{2}(w):]]\\
   &=&-\sqrt{-1}[:X(\varepsilon_{m+1},z_{1})\delta^{*}_{1}(z_{1}):,;X(\varepsilon_{m+1},w)\delta_{2}^{*}(w):]\delta(z_{2}-w)\\
   &=&0
\end{eqnarray*}
The remaining relations follow similarly by Wick's theorem.  This
completes the proof of the theorem.
\end{proof}

\bigskip

\centerline{\bf Acknowledgments} The research is supported by the
National Natural Science Foundation of China (Nos. 11271138,
11531004, 11301393), Zhejiang Natural Science Foundation (grant No.
LY16A010016), Project from Zhejiang province (grant No. FX2014099)
and Simons Foundation (grant no. 198129).


\begin{thebibliography}{ABCD}
\bibitem{BB} S. Berman, Y. Billig, {\em Irreducible representations for toroidal Lie algebras},
J. Algebra 221 (1999), 
188--231.
\bibitem{BCJ} S. Buelk, B. L. Cox, E. Jurisich, {\em A Wakimoto type realization of toroidal $\mathfrak{sl}_{n+1}$},  Algebra Colloq. 19 (2012),
Special Issue no. 1, 841--866. 
\bibitem{FF} A. J. Feingold,  I. B. Frenkel, {\em Classical affine algebras}, Adv. Math. 56 (1985), 117--172.
\bibitem{F} L. Frappat, {\em Vertex operator representation of $OSp(M|N)^{(1)}$}, Int. J. Mod. Phys. A 3 (1988), 2545--2566.
\bibitem{FK} I.B. Frenkel, V. G. Kac, {\em Basic represnetation of affine Lie algebra and dual resonance models},
Invent. Math. 62 (1980), 23--66.
\bibitem{FLM}  I. B.  Frenkel,  J. Lepowsky,  A. Meuraman, {\em Vertex operator algebras
and the monster}, Academic Press, Boston, 1988.
\bibitem{JM} C. Jiang, D. Meng, {\em
Vertex representations for the $\nu+1$-toroidal Lie algebra of type $B_l$}, J. Algebra 246 (2001), 564--593.
\bibitem{JMT} N. Jing, K.~C. Misra, S. Tan, \emph{Bosonic realizations of higher-level toroidal Lie algebras}, Pacific J. Math. 219 (2005),
 285--301.
\bibitem{JMX} N. Jing, K.~C. Misra, C. Xu, {\em Bosonic realization of toroidal Lie algebras of classical types}, Proc. Amer. Math. Soc. 137 (2009), 3609--3618.
 \bibitem{JX1} N. Jing, C. Xu, \emph{Toroidal Lie superalgebras and free field representations}, Contemp. Math.
 623 (2014), 135--153.
 \bibitem{JX2} N. Jing, C. Xu, \emph{Vertex representation of toroidal special linear superalgebras},
 Chin. Ann. Math. Ser. B 36 (2015), 
 427--436.
\bibitem{K1} V. G.  Kac, {\em Vertex algebras for beginners}, Univ. Lect. Ser. 10, Amer. Math. Soc., Providence, 1997.
\bibitem{K2} V. G.  Kac, {\em Infinite-dimensional Lie Algebras}, 3rd ed., Cambridge Univ. Press, Cambridge, 1990.
\bibitem{KWk} V. G. Kac, M. Wakimoto, {\em Integrable highest
weight modules over affine superalgebras and Appell's function},
Comm. Math. Phys. 215 (2001), 631--682.
\bibitem{L} M. Lau, {\em Representations of multiloop algebras}, Pacific J. Math. 245 (2010), 167--184.
\bibitem{MRY} R. E. Moody, S. E. Rao, T. Yokonuma, {\em Toroidal Lie algebras and vertex representations}, Geom. Dedicata 35 (1990), 283--307.
\bibitem{R} S. Eswara Rao,
\emph{Representation of toroidal general linear superalgebras},
Comm. Algebra 42 (2013), 2476--2507.
\bibitem{T} S. Tan, {\em Vertex operator representations for toroidal Lie algebra of type $B_l$}, Comm. Algebra 27 (1999), 3593--3618.
\bibitem{X} X. Xu, {\em Introduction to vertex operator superalgebras and their modules}, Kluwer Academic Publishers, Dordrecht, 1998.
\end{thebibliography}
\end{document}